\documentclass[12pt]{amsart}
\usepackage[utf8]{inputenc}
\usepackage{amsmath,amsthm,amssymb}
\usepackage{xcolor}
\usepackage[shortlabels]{enumitem}
\usepackage{hyperref}
\usepackage{mathtools}
\usepackage{graphicx}
\newtheorem{theorem}{Theorem}[section]

\newtheorem{defn}[theorem]{Definition}
\newtheorem{definition}[theorem]{Definition}

\newtheorem{remark}[theorem]{Remark}
\newtheorem{lemma}[theorem]{Lemma}
\newtheorem{proposition}[theorem]{Proposition}

\mathtoolsset{showonlyrefs}

\newcounter{stepctr}
\newif\ifinsteppage

\newenvironment{steps}{%
	\setcounter{stepctr}{0}%
	\insteppagefalse
}{%
	\par\bigskip
}

\newcommand{\step}[1]{%
	\ifinsteppage
	\par\bigskip
	\fi
	\insteppagetrue
	\refstepcounter{stepctr}%
	\noindent\text{Step \thestepctr:} #1.\par
}

\newcommand{\De}{\Delta}

\newcommand{\dem}{\de^{-O(\e)}}
\newcommand{\dep}{\de^{O(\e)}}
\newcommand{\demm}{\de^{-\e}}
\newcommand{\depp}{\de^\e}

\newcommand{\slp}[1]{\mathrm{slope}(#1)}
\newcommand{\R}{\mathbb{R}}
\newcommand{\T}{\mathbb{T}}

\newcommand{\E}{\mathbb{E}}

\newcommand{\les}{\lessapprox}

\newcommand{\ges}{\gtrapprox}

\newcommand{\cS}{\mathcal{S}}

\newcommand{\cP}{\mathcal{P}}
\newcommand{\cT}{\mathcal{T}}
\newcommand{\cA}{\mathcal{A}}
\newcommand{\cD}{\mathcal{D}}
\newcommand{\cL}{\mathcal{L}}
\newcommand{\N}{\mathbb{N}}

\newcommand{\cU}{\mathcal{U}}

\newcommand{\cB}{\mathcal{B}}

\newcommand{\PP}{\mathbb{P}}
\newcommand{\cH}{\mathcal{H}}

\newcommand{\cI}{\mathcal{I}}

\newcommand{\cQ}{\mathcal{Q}}

\newcommand{\dist}{\mathrm{dist}}

\newcommand{\fm}[1]{\mathfrak{m}(#1)}

\newcommand{\cR}{\mathcal{R}}

\newcommand{\cont}[1]{\mathcal{H}^{#1}_{[\delta,\infty)}}

\newcommand{\de}{\delta}
\newcommand{\e}{\epsilon}

\newcommand{\diam}{\operatorname{diam}}

\newcommand{\dimh}{\dim_\mathrm{H}}

\newcommand{\cns}[2]{N_{{#1}}({#2})}

\newcommand{\cov}[1]{N_\de({#1})}
\newcommand{\cn}[1]{N_\de({#1})}
\newtheorem*{ack*}{Acknowledgment}

\author{Ciprian Demeter}
\author{William O'Regan}

\begin{document}
	\title{New discretised polynomial expander and incidence estimates}

	\thanks{CD is partially supported by the NSF grant  DMS-2349828, WOR is supported in part by an NSERC Alliance grant administered by Pablo Shmerkin and Joshua Zahl}
	\begin{abstract}
    We present two applications of recent developments in incidence geometry. One is a $\delta$-discretised version of a particular `Elekes--R\'onyai' expander problem. The second is an incidence estimate addressing the scenario when both tubes, squares and their shadings satisfy non-concentration assumptions.   
	\end{abstract}
	\maketitle	
	\section{Introduction}
    We start by briefly describing the two main new results in the paper. 
    \subsection{Expander results} 
Let $f:\R^2 \rightarrow \R$ be the polynomial $f(x,y) = x(x+y).$ 
    A very specific application of \cite[Theorem 1.18]{razzahlelekes} gives the following expansion result. 
    \begin{theorem}[Raz--Zahl]\label{thm.razzahl}
        Let $0 < s \leq 1.$  There exists $c_0 = c_0(s) > 0$ so that the following holds for all $\e > 0$ small enough. Let $A, B, \subset [1/2,1]$ be $(\delta,s,\demm)$-sets. Let $\cP \subset A \times B$ be such that $\#\cP > \delta^\epsilon\#A\#B.$  Then
		\begin{equation}
			N_\de(f(\cP)) \gtrsim \# \cP ^{1/2+c_0}.
		\end{equation}
    \end{theorem}
   Here is our first result. It reaches the discretised analogue of the $4/3$ Elekes--R\'onyai exponent in the range $s,t\leq 2/3$. This improvement comes via the recent availability of Szemer\'edi--Trotter strength incidence theorems - see Theorem \ref{thm.product} for the version we use. 
   
   The discretised sum-product problem is closely related, both in its formulation and in the incidence-theoretic methods used to study it. In particular, Ren and Wang \cite{renwang} obtained the $5/4$ exponent by combining the mechanism in \cite{elekes} with the Furstenberg estimate obtained in the same paper.  
	\begin{theorem}\label{thm.ele}
		Let $0 < s,t \leq 2/3$. Let $A\subset [1/2,1]$ be a $(\delta,s,\demm)$-set, and let $B \subset [1/2,1]$ be a $(\delta,t,\demm)$-set. Let $\cP \subset A \times B$ be such that $\#\cP > \delta^\epsilon\#A\#B.$  Then 
		\begin{equation}
			N_\de(f(\cP)) \gtrsim \delta^{O(\epsilon)}\de^{-2(s+t)/3}.
		\end{equation}
	\end{theorem}
When $s=t$ and $\#\cP\approx\delta^{-2s}$, Theorem \ref{thm.ele} gives
\begin{equation*}
    N_\de(f(\cP))\gtrsim \delta^{O(\epsilon)}\delta^{-4s/3}
    \approx \delta^{O(\epsilon)}(\#\cP)^{2/3}.
\end{equation*}
Thus, in terms of $\#\cP$, the exponent $\frac12+c_0$ in Theorem \ref{thm.razzahl} is replaced by $\frac23$. Equivalently, relative to the common one-dimensional input size $\delta^{-s}$, this is the usual $4/3$ Elekes--R\'onyai exponent.

More general statements are derived, where values of $s$ and $t$ outside the above stated range are also considered. The reader is encouraged to view Proposition \ref{prop.ele2}.
	
	In terms of fractal dimension, we obtain the following.
	\begin{theorem}\label{thm.elecont}
		Let $0 < s,t\leq 2/3$. Let $A,B \subset \R$ satisfy $\dimh(A) = s, \dimh(B) = t.$ Then
		\begin{equation}
			\dimh(f(A\times B)) \geq 2(s+t)/3.
		\end{equation}
	\end{theorem}
	It is unclear what the sharp bounds ought to be. In the discrete case \cite{solyzahlprox} obtain the exponent $3/2$, improving $4/3$ obtained by \cite{razsharirsolyelekes}. The arguments in the former use algebraic methods, which have found limited power in the discretised setting thus far. This does not mean it would be unreasonable to hope that the exponent $3/2$ may be able to be obtained in the discretised setting - for some range of $s$ and $t.$ An exponent of 1 may be the correct answer in some range of $s,t.$

 The fact that the polynomial $f$ mixes additive and multiplicative structure is important. If $f$ could be written in the form 
	\begin{equation}\label{eq.special}
		f(x,y) = h(p(x) + q(y)) \text{ or } f(x,y) = h(p(x)q(y)),
	\end{equation}
	where $h,p,q$ are real-valued univariate polynomials, then it is not expected that Theorem \ref{thm.razzahl} holds for these polynomials. 
    
    This rules out examples such as $x+y$ and $xy,$ for which it is known that no growth can be expected in general. In fact, Raz and Zahl proved that the conclusion of Theorem \ref{thm.razzahl} holds provided that $f$ cannot be written in the form \eqref{eq.special}. 

Problems of this type are sometimes called `Elekes--R\'onyai' problems after \cite{elekesronyai}. There, they showed that for polynomials for which there exist finite $A,B \subset \R$ with $\#A = \#B = N$ and   
\begin{equation}\label{eq.growth}
\# (f(A\times B)) \sim N,
\end{equation}
then $f$ must satisfy \eqref{eq.special}. In \cite{razsharirsolyelekes}, Raz--Sharir--Solymosi revisited this problem, and showed that polynomials which cannot be written in the form \eqref{eq.special} must exhibit very significant growth over \eqref{eq.growth}. In particular, for such a polynomial $f,$ and all finite $A,B \subset \R$ with $\# A = \#B= N,$ we have
\begin{equation}
    \# (f(A\times B)) \gtrsim N^{4/3}.
\end{equation}
As mentioned, the exponent $4/3$ has now been improved to $3/2$ in \cite{solyzahlprox}.

	To prove Theorem \ref{thm.ele}, we turn the problem into obtaining an upper bound on a certain collection of incidences, with the aim of applying Theorem \ref{thm.product} - this approach was used in \cite{razsharirsolyelekes}. This method, while giving strong bounds on the growth under $f,$ only turns into an incidence problem for balls and tubes for very specific choices of $f.$ It is likely that Theorem \ref{thm.ele} generalises to other real-valued semi-diagonal quadratic forms, e.g.~ $x^2 - 10xy,$ but we do not pursue this here.
    	
    \subsection{An incidence estimate} 
    Here is our second new result. The reader should compare these bounds with the trivial one
    $$I(\T,\PP)\lesssim \delta^{-s}(\# \T\# \PP)^{1/2}.$$
    \begin{theorem}
		\label{huufwf[gmiu9i96=y]new}
		Assume $0\le s\le \frac23$.	
		Consider a collection $\T$ of $\delta$-tubes and a collection $\PP$ of $\delta$-squares such that both $\T$ and $\PP$ are $(\delta,2s)$-KT sets, and both $Y(T)=\{p\in\PP:\;p\cap T\not=\emptyset\}$ and $Y'(p)=\{T\in\T:\;p\cap T\not=\emptyset\}$ 
		are $(\delta,s)$-KT sets, for each $T\in\T$ and each $p\in\PP$. Write
		$$I(\T,\PP)=\sum_{T\in\T}\#Y(T)=\sum_{p\in\PP}\#Y'(p).$$
		Then if $s\le \frac12$
		$$
			I(\T,\PP)\les \delta^{-\frac{3s}{4}}(\#\T\#\PP)^{1/2},
		$$
		while if $s>\frac12$
		$$I(\T,\PP)
			\les \delta^{-(s-\frac{s^2}{2})}(\#\T\#\PP)^{1/2}
            +\delta^{-\frac{s}{2}}(\#\T+\#\PP).
		$$
	\end{theorem}
    There is no quasi-product structure in either $\T$ or $\PP$, so Theorem \ref{thm.product} is not useful for this result. Instead, we will use Theorem \ref{completeWW}. Since the latter requires a two-ends condition, we have to identify large collections of both $\T$ and $\PP$ that have this property. Enforcing the simultaneity of these refinements leads to delicate considerations, see Remark \ref{lkijiourtiouiruythu98}. 
    
    It is unclear whether the exponents $3s/4$ and $s-\frac{s^2}{2}$ are sharp. This theorem was originally motivated by the energy problem for the parabola discussed in \cite{doprod}. Indeed, counting the number of additive 6-tuples can be rephrased as an incidence geometry problem between tubes and squares. In the end, a less sophisticated incidence estimate was used in \cite{doprod}, that took advantage of the additional rectangular KT structure present there. Nevertheless, we believe that both Theorem \ref{huufwf[gmiu9i96=y]new} and the method employed in its proof are  likely to find further use.
        
        \subsection{Definitions and notation} We use the notation from \cite{doprod}, which we recall below.
        
        The   $\delta$-squares in $[0,1]^d$ are $$\cD_\delta=\{[n_1\delta,(n_1+1)\delta]\times \cdots \times [n_d\delta,(n_d+1)\delta]:\;0\le n_1,\ldots n_d\le \delta^{-1}-1\}.$$

	We use $|\;\;|$ to denote Lebesgue measure and $\#$ for the counting measure. For $\de > 0$ we use $\cn{S}$ to denote the $\de$-covering number of $S$.
	
	If quantities $A,B$ depend on $\delta$, we write $A \lesssim B$ to mean there exists $C > 0$ so that $A(\delta) \leq CB(\delta)$ for all $\delta > 0.$  We use $\lesssim_\e C$ if we wish to emphasise that the suppressed constant depends on the parameter $\e.$ We write $A\sim B$ if $A\lesssim B$ and $B\lesssim A$. 
	
	We write $A\les B$ to mean that $A\lesssim_\upsilon \delta^{-\upsilon} B$ for each $\upsilon>0$. Also, $A\approx B$ means that both $A\les B$ and $B\les A$ hold.
	
	For a set $A \subset \R^d$ and $\delta >0$ we let  $A_\de$ denote a $\de$-separated subset of $A$ with $\#A_\de \sim \cn A.$ We denote the Hausdorff content of $A$ at dimension $0 \leq s \leq d$ at scale $\delta > 0$ by
	\begin{equation}
		\cH_{[\delta,\infty)}^s(A) := \inf\bigg\{\sum_{U \in \cU} \diam(U)^s : A \subset \bigcup_{U \in \cU}U \text{ and } \diam(U) \in \{\de,2\de,\ldots\} \text{ for all } U \in \cU\bigg\}.
	\end{equation}
	All sets $A$ that we consider will have diameter $O(1)$, so the sets $U$ in each cover will have diameter $\diam(U)\lesssim 1$.

	\begin{definition}
		Let $0 \leq s \leq d$ and  $C, \delta >0.$ We say that $A \subset \R^d$ is a $(\delta,s,C)$-\textit{set} if 
		\begin{equation}
			\cn{A \cap B(x,r)} \leq Cr^s\cn{A}\text{ for all } x \in \R^d,\; r \geq \delta.
		\end{equation}
		In some literature these sets are called Frostman sets.
		
		We say that $A \subset \R^d$ is a $(\delta,s,C)$-\textit{KT set} if 
		\begin{equation}
			\cn{A \cap B(x,r)} \leq C(r/\delta)^s \text{ for all } x \in \R^d,\; r \geq \delta.
		\end{equation}
		In some literature these sets are called Katz-Tao sets.

		If $C \sim 1$ then we will often drop the $C$ from the notation and simply refer to $(\delta,s)$-sets and $(\de,s)$-KT sets.
	\end{definition}
    \begin{defn}\label{KK}
		Given $K_1,K_2\ge 1$ and $s,d\in[0,1]$ we call a  collection $\T$ of $\delta$-tubes a $(\delta,s,d,K_1,K_2)$-quasi-product set if its direction set $\Lambda$ is a $(\delta,s,K_1)$-KT set and for each $\theta\in\Lambda$, $\T_\theta$ is a $(\delta,d,K_2)$-KT set. Here $\T_\theta$ is the collection of tubes which point in the direction $\theta.$
        \end{defn}
A shading  of $\T$ is a map $Y$ on $\T$ such that for each $T\in\T$,  $Y(T)$ is collection of pairwise disjoint  $\delta$-squares in $\cD_\delta$ intersecting $T$.
	We let
	$$Y(\T)=\cup_{T\in\T}Y(T),$$
	and write $\#Y(T)$, $\#Y(\T)$ for the total number of (distinct) $\delta$-squares in the collections $Y(T)$ and $Y(\T)$, respectively.
	We caution that $p\in Y(\T)$ and $p\cap T\not=\emptyset$ does not necessarily imply that $p\in Y(T)$.
	
	For $0<\epsilon_2 < \epsilon_1\ll 1$ we say that a shading is $(\epsilon_1,\epsilon_2)$-\textit{two-ends} if 
	\begin{equation}
		\#(Y(T) \cap B(x,\delta^{\epsilon_1})) \leq \delta^{\epsilon_2}\#Y(T) \text{ for all } x \in \R^2, \;T \in \T.
	\end{equation}

    		\subsection*{Acknowledgements}
        WOR thanks Pablo Shmerkin and Joshua Zahl for introducing him to the Elekes--R\'onyai problem. We thank the anonymous referee for comments greatly improving the quality of the exposition.
	
	\section{Auxiliary results}

We will repeatedly use the point-line duality. There is a transformation that maps points to lines and lines to points, in a way that incidences are preserved. More precisely, if $P=(a,b)$ is mapped to the line $l_P:\;y=ax-b$, and the line $L:\;y=mx+c$ is mapped to the point $P_L$ with coordinates $(m,-c)$, then $P\in L\iff P_L\in l_P$.
\medskip

Let $\sigma = \sigma(s,d) := \min\{s+d,2-s-d\}.$ We recall \cite[Theorem 1.3]{doprod}. The case $s+d=1$ was proved earlier in \cite{demwangszem}.
	\begin{theorem}
		\label{maininct}\label{thm.product}
		Let $0<s\leq d \leq 1$.
		Assume $\cT$ is a $(\delta,s,d,K_1,K_2)$-quasi-product set. Consider a shading $Y$ of $\cT$ such that $Y(T)$ is a $(\delta,\sigma,K_3)$-KT set for each $T\in\T$. Write
        $$Y(\cT)=\cup_{T\in\cT}Y(T).$$       
        Then  we have
		$$\cI(\cT,Y) \les K_3^{\frac13}(K_1K_2)^{2/3} (\delta^{-s-d}\#\cT)^{1/{3}}\# Y(\cT)^{2/3}.$$	
	\end{theorem}
    Here 
    $$\cI(\cT,Y) = \sum_{T \in \cT}\# Y(T).$$
    
    This result was derived as a consequence of the following general incidence estimate of Wang and Wu. This version with arbitrary $K_1$ was verified in \cite{doprod}.
\begin{theorem}[\cite{wangwufurst}]
		\label{completeWW}
		Let $\T$ be a $(\delta,t,K_1)$-KT set for some $t\in(0,2)$, and let $0<\epsilon_2<\epsilon_1$.
		Let $\sigma=\min(t,2-t)$.
        Let $Y$ be a shading of $\T$.
		Assume that each $Y(T)$ is a $(\delta,\sigma,K_2)$-KT set with cardinality $N$,  and assume that it is also $(\epsilon_1,\epsilon_2)$-two ends.
		
		Then for each $\epsilon>0$
		$$   N^{1/2}\delta^{t/2}\sum_{T\in\T}\#Y(T)\le C(\epsilon, \epsilon_1,\epsilon_2)K_1K_2^{1/2}\delta^{-O(\epsilon_1)-\epsilon}\#Y(\T).
		$$
	\end{theorem}

The below definition is standard, and can be found at \cite[Definition 2.3]{demwangszem}, for example.
	\begin{definition}\label{def.uniform}
		Let $\epsilon >0.$ Let $T_\epsilon$ satisfy $T_\epsilon^{-1}\log (2T_\epsilon) = \epsilon.$ Given $0 < \delta \leq 2^{-T_\epsilon},$ let $m$ be the largest integer such that $mT_\epsilon \leq \log (1/\delta).$ Set $T := \log(1/\delta)/m.$ Note that $\delta = 2^{-mT}.$ We say that a $\delta$-separated $A \subset \R^d$ is \textit{$\e$-uniform} if for each $\rho = 2^{-jT}, 0 \leq j \leq m$ and each $P,Q \in \cD_\rho(A),$ we have
		\begin{equation}
			\#(A \cap P) \sim \#(A \cap Q).
		\end{equation}
		A similar definition may be given if $A$ is a collection of $\delta$-balls.
	\end{definition}
	The lemmata below will be used without reference. Below is \cite[Lemma 2.4]{demwangszem}.
	\begin{lemma}
		Let $\epsilon > 0.$ Let $A \subset \R^d$ be a $\delta$-separated $\e$-uniform set. For each $\delta \leq \rho \leq 1$ and each $P,Q \in \cD_\rho(A)$ we have 
		\begin{equation}
			C_\epsilon^{-1}\#(A \cap P) \leq \#(A \cap Q) \leq C_\e \#(A \cap P).
		\end{equation}
	\end{lemma}
	Below is \cite[Lemma 2.15]{orpshabc}. It tells us that every set contains a dense uniform subset. 
	\begin{lemma}\label{lem.unifsubset}
		Let $\epsilon > 0$ and suppose that $\delta >0 $ is small enough in terms of $\e.$ Let $A \subset \R^d$ be a $\delta$-separated set. Then there is an $\e$-uniform $A_0 \subset A$ with $\#A_0 \gtrsim \delta^{\epsilon}\#A.$ 
	\end{lemma}

	\section{Proofs of expander results}\label{sect.elekesronyai}
	Throughout this section let $f:\R^2 \rightarrow \R$ be the polynomial $f(x,y) = x(x+y).$ We restate Theorem \ref{thm.ele} for a larger range of $s$ and $t.$  By analysing four cases depending on how $s$ and $t$ compare to $1/2$, we find that when $s,t\le 2/3$ at least one of the inequalities $s \leq \min\{2t, 2-2t\}$,  $t \leq \min\{2s,2-2s\}$ must be true. This shows that Theorem \ref{thm.ele} is a particular case of Proposition \ref{prop.ele2}.
    	\begin{proposition}\label{prop.ele2}
		Let $\e >0$, $ 0 < s,t <1.$  Let $A\subset [1/2,1]$ be a $(\delta,s,\demm)$-set and let $B \subset [1/2,1]$ be a $(\delta,t,\demm)$-set. Let $\cP \subset A \times B$ be such that $\#\cP > \delta^\epsilon\#A\#B.$  If either $s \leq \min\{2t, 2-2t\}$ or $t \leq \min\{2s,2-2s\},$ then
 \begin{equation}
 	\cn {f(\cP)} \gtrsim \dep  \de^{-\frac{2(s+t)}{3}}.
 \end{equation}

\end{proposition}
This will follow from the proposition below, formulated for KT sets. 
	\begin{proposition}\label{prop.ele}
		Let $\e >0$, $0 < s,t <1.$  Let $A\subset [1/2,1]$ be a $(\delta,s,\demm)$-KT-set and let $B \subset [1/2,1]$ be a $(\delta,t,\demm)$-KT-set. Let $\cP \subset A \times B$ be such that $\#\cP > \delta^\epsilon\#A\#B.$  If either $s \leq \min\{2t, 2-2t\}$ or $t \leq \min\{2s,2-2s\},$ then
 \begin{equation}
 	\cn {f(\cP)} \gtrsim \dep  \de^{2(s+t)/3}\#\cP^{4/3}.
 \end{equation}
		
	\end{proposition}
    Throughout this section we will implicitly use the fact that $A,B \subset [1/2,1]$ without mention. 
\begin{proof}[Proof of Proposition \ref{prop.ele2} assuming Proposition \ref{prop.ele}] We prove this when $s \leq \min \{2t,2-2t\}.$ The other case follows by symmetry. The proof follows a standard argument.
\medskip

\begin{steps}
    \step{Reduction to the case when $A$ and $B$ are uniform}
We find dense subsets of $A$ and $B$ which are $\e$-uniform but still retain a large fraction of $\cP.$ Let $\pi_x:\cP\to A$ be the projection. By dyadic-pigeonholing, we may find a dyadic integer $N$ and some $A_0 \subset A$ so that for all $a \in A_0$ we have $\# \pi_x^{-1}(a) \sim N$ and 
\begin{equation}
    \# \bigg(\bigcup_{a \in A_0}(\pi^{-1}_x(a) \cap \cP)\bigg) \approx \#\cP.
\end{equation}
Since $N \leq \#B,$ it follows that $\#A_0 \geq \de^{-\e}\#A.$ Therefore $N \sim \#\cP/\#A_0.$ We now find $A_1 \subset A_0$ which is $\e$-uniform with $\#A_1 \geq \de^{10\e}\#A.$ Set 
$$\cP_0 := \bigcup_{a \in A_1}(\pi^{-1}(a)\cap \cP) \subset A_1 \times B.$$ In particular,
\begin{equation}
    \#\cP_0 \sim \#A_1 N \sim \#\cP \#A_1 / \# A_0 \geq \de^{10\e}\#\cP > \de^{100\e}\# A_1 \# B.  
\end{equation}
We perform a symmetric analysis to $B,$ but now replacing $\cP$ with $\cP_0;$ we find an $\e$-uniform $B_1 \subset B$ with $\#B_1 \geq \de^{10\e}\#B.$ This outputs some $\cP_1 \subset A_1 \times B_1,$ with 
\begin{equation}
    \#\cP_1 \geq \de^{10\e}\cP_0 > \de^{1000\e}\#A_1\#B_1.
\end{equation}
\step{Decomposition of $A_1,B_1,\cP_1$}
We hide constant factors of $\e.$ Observe how $A_1$ is a $(\de, s, \dem \#A_1 \de^s)$-KT-set, and how  $B_1$ is a $(\de, t, \dem \#B_1 \de^t)$-KT-set. Apply \cite[Lemma 2.8]{demwangszem} to find integers $S \sim \dem \#A_1 \de^s$ and $T \sim \dem \#B_1 \de^t$ and a partition of $A_1$ and $B_1$ into disjoint sets
\begin{equation}
    A^1\ldots A^S,
\end{equation}
where each is a $(\de,s,\demm)$-KT-set, and
\begin{equation}
    B^1 \ldots B^T,
\end{equation}
each of which is a $(\de,t,\demm)$-KT-set, respectively. We find a pair $A^i, B^j$ for which $\cP_1$ is dense in $A^i \times B^j.$ This is straightforward pigeonholing; since 

    \begin{equation}
     \sum_{i,j}\#(\cP_1 \cap A^i \times B^j)  \geq \dep \sum_{i,j} \#A^i \#B^j, 
    \end{equation}
    there must be a pair $i,j$ so that
    \begin{equation}
    \#(\cP_1 \cap A^i \times B^j) \geq \dep \# A^i \# B^j.    
    \end{equation}
    Set $A_2 := A^i, B_2 := B^j, \cP_2 := (\cP_1 \cap A_2 \times B_2).$
    
    \step{Completion of the proof}
     This triple $(A_2,B_2,\cP_2)$ satisfies the hypothesis of Proposition \ref{prop.ele}. Since $\cP_2 \subset \cP$ we have
     \begin{equation}
         \cn{\cP} \geq \cn {\cP_2}\gtrsim \de^{2(s+t)/3}\# \cP_2^{4/3} \gtrsim \dep \de^{-2(s+t)/3},
     \end{equation}
     as required. 
    \end{steps}
\end{proof}
\subsection{Proof of Proposition \ref{prop.ele}}
			\begin{proof}

	\begin{steps}
	\step{Turning the problem into an incidence problem} In the spirit of that in the discrete case \cite{razsharirsolyelekes}, we turn the problem of finding a lower bound for $\cn {f(\cP)}$ into finding an upper bound for incidences between points and tubes.
	
	 Let $X$ be a  $\delta/2$-separated subset of $f(\cP)$ with $\cov {f(\cP)}\sim \#X$.
	By Cauchy--Schwarz we have 
	\begin{equation}\label{eq.cs} 
		\sum_{(a,b), (a',b') \in \cP} 1_{|a(a+b) - a'(a'+b')| \leq \de} \gtrsim 	\# \cP^2 / \# X.
	\end{equation}
	We need to prove an appropriate upper bound on the left hand side. 
	Write $$\cA := \{(a,a') \in A^2 : \text{there exists } (b,b') \in B^2 \text{ such that } (a,b),(a',b') \in \cP\},$$
	and
	$$\cB := \{(b,b') \in B^2 : \text{there exists } (a,a') \in A^2 \text{ such that } (a,b),(a',b') \in \cP\}.$$ Let $F(x,y,z,w) = (x,z,y,w).$ Since
    $$F(\cP^2) \subset \cA \times \cB$$ from \eqref{eq.cs} we have
	\begin{equation}\label{eq.cs2} 
		\sum_{(a,a') \in \cA} \sum_{(b,b') \in \cB} 1_{|a(a+b) - a'(a'+b')| \leq \de} \gtrsim 	\# \cP^2 / \# X.
	\end{equation}
	Since $\#\cP > \depp \#A \#B$ and 
	$$
	\#\cP^2 \leq \#\cA \# \cB \leq \#\cA \# B^2, 
	$$
	we have $$\# \cA \geq \# \cP^2/ \# B ^2 \gtrsim \de^{O(\e)}\# A^2.$$ The bound $$\# \cB \geq \# \cP^2/ \# A ^2 \gtrsim \de^{O(\e)}\# B^2$$
	follows by symmetry.

	Let $(a,a') \in \cA.$ Write
	\begin{equation}\label{eq.linedef}
		l_{a,a'} := \{y = (a/a') x + (a^2 -a'^2)/a'\},
	\end{equation}
	and 
	$$
	\cL:= \{l_{a,a'} : (a,a') \in \cA\}.
	$$
	Let $\cT$ be the $2\de$-neighbourhoods of the lines in $\cL$; we refer to them as tubes. Two or more tubes may be essentially the same, we allow multiplicity and will quantify it momentarily. Since
	\begin{equation}
		|a(a+b) - a'(a'+b') | \leq \de
	\end{equation}
	implies that
	\begin{equation}
		|b' - \big((a/a')b + (a^2 - a'^2)/a'\big)| \leq 2\de,
	\end{equation}
    and thus $\dist((b,b'),l_{a,a'})\le 2\de$,
	we may  rewrite \eqref{eq.cs2} as
	\begin{equation}\label{eq.cs3}
		\cI(\cB, \cT) \gtrsim \#\cP^2/\# X.
	\end{equation}
    Here $$\cI(\cB,\cT) = \sum_{T \in \cT}\sum_{b \in \cB}1_{b \in T}.$$
    The result will then follow from an appropriate upper bound on $\cI(\cB,\cT).$
\step{Creating a partition of $\cT$}
We must evaluate the multiplicity of each tube. Let $l_{u,u'}, l_{v,v'} \in \cL.$ These lines will give rise to essentially the same tube if we have both
\begin{equation}\label{eq.one}
	\big|\frac u{u'} - \frac v {v'}\big| \leq \de, 
\end{equation}
and 
\begin{equation}\label{eq.two}
	\big|(\frac{u^2}{u'} - u') - (\frac{v^2}{v'} - v')\big| \leq\de. 
\end{equation}
Otherwise the two tubes will be referred to as distinct. 
For a given tube $T$ with axial line $l_{u,u'}$ its multiplicity is therefore
\begin{equation}
	\fm T := \# \{(v,v') \in \cA : \text{ \eqref{eq.one} and \eqref{eq.two} hold} \}
\end{equation}

Fix a dyadic $\de \leq \De \leq 2.$ Let $\cT_\De$ be a maximal collection of distinct tubes $l_{a,a'}+O(\delta)$ in $\cT$ whose axial lines have slopes $m=a/a'$ satisfying 
$$\de/\De \leq |m-1| < 2\de/\De, \text{ if }\Delta\le 1$$
and 
$$|m-1| < \de, \text{ if }\Delta=2.$$

For a dyadic integer $N\geq 1$ we define 
$$\cT_{\De,N} := \{T \in \cT_\De : N \leq  \fm T < 2N\}.$$
Therefore we may decompose the collection  $\cT_{dis}$ of distinct tubes (a maximal collection of distinct tubes in $\cT$) into the disjoint union
\begin{equation}
    \cT_{dis} = \bigcup_{l=-1}^{\log 1/\de} \bigcup_{n=0}^{\lceil\log \# A\rceil }  \cT_{2^{-l},2^n}.
\end{equation}
We deal with each $\cT_{\De,N}$ separately. 
\step{Properties of $\cT_{\De,N}$}

We wish to have a bound for the number of tubes in $\cT_{\De,N}.$ We show that
\begin{equation}\label{eq.boundonpreT}
\{(u,u') \in \cA: |u/u' -1| \lesssim \de/\De\} \lesssim \demm \# A \De^{-s}.
\end{equation}
Simply note that 
\begin{equation}
	\big|\frac{u}{u'} - 1\big| \lesssim \de/\De
\end{equation}
implies that 
\begin{equation}
	|u - u'| \lesssim \de/\De.
\end{equation}
So with $u$ fixed, since $A$ is a $(\de,s,\demm)$-KT-set, the total number of choices for $u'$ is  $\lesssim \demm \De^{-s}.$ Ranging over the $\#A$ possible values of $u$ gives us \eqref{eq.boundonpreT}. Immediately, we obtain
\begin{equation}\label{eq.boundonTN}
	\#\cT_{\De,N} \lesssim \demm  \frac{\De^{-s}\# A}{N},
\end{equation}
since this number is just the bound in \eqref{eq.boundonpreT}, but observe that we have overcounted by a factor of $N.$
\step{Decomposing the incidences}
We have 
\begin{align}
	\cI(\cB,\cT) &\sim  \sum_{l=-1}^{\log 1/\de} \sum_{n=0}^{\lceil\log \# A\rceil } \cI(\cB,\cT_{2^{-l},2^n})2^n\\
    &=\sum_{l=3}^{\log 1/\de} \sum_{n=0}^{\lceil\log \# A\rceil } \cI(\cB,\cT_{2^{-l},2^n})2^n + \sum_{l=-1}^{2}\sum_{n=0}^{\lceil\log \# A\rceil } \cI(\cB,\cT_{2,2^n})2^n.\label{eq.decomp}
\end{align}
The second summand in \eqref{eq.decomp} is bounded by
\begin{equation}
   \lesssim \cI(\cB, \cT_{0})\# A \lesssim \# A \# B \lesssim \dem \# \cP,
\end{equation}
where 
\begin{equation}
    \cT_{0} := \{T=l_{a,a}+O(\delta) \}. 
\end{equation}
Due to a bit of extra $O(\delta)$ fattening that we allow, there is (essentially) only one tube in $\cT_{0}$, as all lines $l_{a,a}$ are the same. The factor $\#A$ accounts for the multiplicity of this tube. This tube only contains the points $\{(b,b),\;b\in B\}$.

The aim now is to bound  $\cI(\cB,\cT_{2^{-l},2^n})$ for each $l \geq 3.$ 
\step{Defining a shading on each $(b,b') \in \cB$}
Fix a dyadic  $N$ and a dyadic $ \De \leq 1/8.$ For $(b,b') \in \cB$ we shade this point with the tubes that are incident to it:
\begin{equation}\label{eq.shadingdef}
	Y((b,b')) := \{T \in \cT_{\De,N} : (b,b') \in T \}.
\end{equation}
We claim that this shading is a $(\de,s,K_3)$-KT-set, where $K_3 \lesssim \demm \frac1N(\De/\de)^s.$

Let $$\de \leq r \leq \frac{\de}{8\De}.$$ 
Fix an interval 
$$(m-r/2,m+r/2)\subset\{m':\;|m'-1|\sim \frac{\de}{\De}\}.$$
The aim is to show that
\begin{equation}\label{eq.shadingbound}
	\#\{T \in \cT_{\De,N} : (b,b') \in T \text{ and } \slp T \in (m-r/2, m+r/2)\} \leq  K(r/\de)^s.
\end{equation} 
The same inequality with $K_3$ replaced with $O(K_3)$ will follow for values of $r$ larger than $\frac{\delta}{8\De}$ and smaller than $\frac{100\de}{\De}$.

Let $T$ be contained in the set above, and let it have axial line $l_{x,y}.$ We have 
$$\de/(2\De)\le |x-y|\le 4\de/\De$$
and thus
$$\de/(2\De)\le |x^2-y^2|\le 16\de/\De.$$
Write $X := b'm - b$ and $Y := bm - b'.$ 
Since $(b,b') \in T$ we must have
\begin{equation}\label{eq.equationsolve}
	|x^2 + bx - y^2 - b'y| \leq \de.
\end{equation}
This implies that
$$|(b-b')x|\le \delta+|x^2-y^2|+|b'(x-y)|,$$
thus
\begin{equation}
\label{bhrgfygrfuygeruyf}
|b-b'|\lesssim \de/\De.
\end{equation}
Since $$x/y \in (m-r/2, m+r/2),$$ 
we may write 
\begin{equation}\label{eq.eqinx}
x = my + R
\end{equation}
for some $|R| \leq r.$ 

Our first aim is to show that $\de/\De \sim |Y| \gtrsim |X|.$ It is here that we use the fact that $\De \leq 1/8. $ We briefly concern ourselves with constants. 
From \eqref{eq.equationsolve}, we have
\begin{align}
    \de &\geq |x^2 + bx - y^2 - b'y|\\
    &\geq \big||x^2-y^2| - |bx-b'y|\big|.
\end{align}
Since $16\de/\De\ge |x^2-y^2|\ge \de/(2\De)\ge 4\de$,  we have that$$20\de/\De\ge |x^2-y^2|+\de\ge |bx-b'y|\ge |x^2-y^2|-\de\ge \frac34|x^2-y^2|\ge 3\de/(8\De).$$
By \eqref{eq.eqinx}  we have
$$
 |bx-b'y|=|y(bm-b') +bR|=|yY+bR|. 
$$
Therefore, since $|bR|\le 2r\le \de/(4\De)\le \frac23|bx-b'y|$, we have that
$
    |yY| \sim \de/\De,
$
and thus
\begin{equation}
    |Y| \sim \frac{\de}{\De}.
\end{equation}
Using \eqref{bhrgfygrfuygeruyf}
we have

\begin{equation}
|X| = |b-b'm| \leq |b-b'| + |b'(1-m)| \sim \de/\De.   
\end{equation}

We next proceed to verify \eqref{eq.shadingbound}. Using \eqref{eq.eqinx}, \eqref{eq.equationsolve} becomes
\begin{align}
&(my+R)^2 + b(my+R) - y^2 - b'y = O(\de)\\
\iff &(m^2-1)y^2 + (2Rm + bm-b')y +bR + R^2 - O(\de) =0.
\end{align}
Solving for $y$ using the quadratic formula gives
\begin{align}\label{eq.solvefory}
    y_{\pm} &= \frac{-2mR - Y \pm \sqrt{D}}{2(m^2-1)}
\end{align}
where 
\begin{equation}
    D := 4R^2 - 4RX + Y^2 + |m^2-1|O(\de)=4R^2 - 4RX + Y^2 +O(\de^2/\De).
\end{equation}
Therefore
\begin{equation}
    |y_{\pm} + \frac{Y}{2(m^2-1)}\mp\frac{\sqrt{D}}{2(m^2-1)}| \lesssim \frac{R}{|m-1|}\lesssim \frac{r\De}{\de}.
\end{equation}
Also,
\begin{align}
    |\sqrt{D} - |Y|| &= \frac{|D - Y^2|}{\sqrt{D} + |Y|}\\
    &\lesssim \frac{R^2 + R|X| +O(\de^2/\De) }{|Y|}\\
    &\lesssim R +O(\de)\\
    &\lesssim r,
\end{align}
where we use the fact that $|X|/|Y| = O(1)$ and $R \lesssim \de/\De \sim |Y|.$

For the solution $ y_-$ to \eqref{eq.solvefory}, we have 
\begin{align}
|y_{-} + \frac{Y + |Y|}{2(m^2-1)} |\leq |y_{-} + \frac{Y+\sqrt{D}}{2(m^2-1)}| + |\frac{|Y|-\sqrt{D}}{2(m^2-1)} |\lesssim \frac{r\De}{\de} .
\end{align}
Likewise, 
\begin{equation}
    |y_{+} + \frac{Y - |Y|}{2(m^2-1)}|\leq |y_{+}+ \frac{Y-\sqrt{D}}{2(m^2-1)}| + |\frac{\sqrt{D}- |Y|}{2(m^2-1)} |\lesssim \frac{r\De}{\de}.
\end{equation}
In either case, since 
\begin{equation}
    \frac{Y \pm |Y|}{2(m^2-1)}
\end{equation}
is constant, it follows that $y$ is constrained to an interval of length $\lesssim r\De/\de.$ Since $y$ ranges in $A,$ and $A$ is $(\de,s,\demm)$-KT, the number of admissible $y$ is 
\begin{equation}
    \lesssim  \de^{-\e}(\De/\de)^s (r/\de)^s.
\end{equation}
Once $y$ has been fixed, the number of those $x$ which satisfy \eqref{eq.equationsolve} is $\lesssim 1.$ Since each tube in $\cT_{\De,N}$ has multiplicity $N,$ the number of such slopes is therefore
\begin{equation}
    \lesssim \de^{-\e}\frac1N (\De/\de)^s (r/\de)^s, 
\end{equation}
thus verifying \eqref{eq.shadingbound}.
\step{Completing the proof}
Since $s \leq \min \{2t,2-2t\}$ we may apply Theorem \ref{thm.product} in its dual form. More precisely, we view $\cB$ as a $(\delta,t,t,\de^{-O(\e)},\de^{-O(\e)})$-quasi-product set. Also $\sigma=\min(2t,2-2t)$ is $\ge s$, and thus our shadings are $(\de,\sigma,K_3)$-KT-sets, with the same $K_3 \lesssim \demm \frac1N(\De/\de)^s$ as before.   We have
\begin{equation}
	\cI(\cB,\cT_{\De,N}) \lesssim \dem \frac{  (\De/\de)^{s/3}\de^{-2t/3}\#\cT_{\De,N}^{2/3}\#\cB^{1/3}}{N^{1/3}}.
\end{equation}
Using \eqref{eq.boundonTN} we then obtain
\begin{equation}
	\cI(\cB,\cT_{\De,N}) \lesssim \dem  \frac{  (\De/\de)^{s/3}\Delta^{-2s/3}\de^{-2t/3}\#A^{2/3}\# \cB^{1/3}}{N}.
\end{equation}
Using \eqref{eq.decomp} we therefore obtain
\begin{equation}
	\cI(\cB,\cT) \lesssim \dem (\#\cP+\sum_{l=3}^{\log 1/\de} \sum_{n=0}^{\lceil\log \# A\rceil }(2^{-l}/\de)^{s/3}2^{2ls/3}\de^{-2t/3}\#A^{2/3}\#\cB^{1/3}).
\end{equation}
After summing over $n$ and then over $l$ we obtain
 \begin{equation}
 	\cI(\cB,\cT) \lesssim \dem (\#\cP+ \de^{-2(s+t)/3}\#A^{2/3}\#\cB^{1/3}) \lesssim \dem \de^{-2(s+t)/3}\#\cP^{2/3}.
 \end{equation}
 Returning to \eqref{eq.cs3} we obtain
 \begin{equation}
 	\cn {f(\cP)} \gtrsim \dep  \de^{\frac{2(s+t)}3}\#\cP^{4/3}.
 \end{equation}
 Arguing symmetrically with the roles of $A$ and $B$ reversed, the case $t \leq \min\{2s,2-2s\}$ gives us the same result.
 
	\end{steps}
	
	\end{proof}
  
	\subsection{Proof of Theorem \ref{thm.elecont}} We now prove Theorem \ref{thm.elecont} which in light of Proposition \ref{prop.ele2}, we state in the more general form:
	\begin{proposition}\label{prop.elecontthm}
		Let $0 < s,t <1.$ Let $A,B \subset \R$ satisfy $\dimh(A) = s, \dimh(B) = t.$ 
		If either  $s \leq \min\{2t,2-2t\}$ or $ t \leq \min\{2s,2-2s\}$ then 
        \begin{equation}
			\dimh(f(A\times B)) \geq 2(t + s)/3.
		\end{equation}
	\end{proposition}
	
	The proof uses the by now standard mechanism which reduces a continuous statement to a discretised statement. We include it for completeness. See also \cite[Theorem 1.20]{razzahlelekes}.
	
	Proposition \ref{prop.elecontthm} will follow from the following multi-scale estimate.
	\begin{proposition}\label{prop.elecont}
		Let $\e >0, 0 < s,t < 1.$ Let $A\subset [1/2,1]$ be a $(\delta,s,\demm)$-KT-set with $\#A > \delta^{-s+\e}$ and let $B \subset [1/2,1]$ a $(\delta,t,\demm)$-KT-set with $\# B > \delta^{-t+\e}.$ If either  $s \leq \min\{2t,2-2t\}$ or $t \leq \min\{2s,2-2s\},$ then 
		\begin{equation}
			\cont{2(t + s)/3}(f(A\times B)) \gtrsim \de^{O(\e)}.
		\end{equation}     
	\end{proposition}
	Given Proposition \ref{prop.elecont} we prove Proposition \ref{prop.elecontthm}.
	\begin{proof}[Proof of Proposition \ref{prop.elecontthm}]
		Let $\e >0.$ Suppose $s \leq \min\{2t,2-t\},$ the other case follows by symmetry. We suppose that $A,B \subset [1/2,1],$ are closed, and $\cH^s_\infty(A)$,
        $\cH^t_\infty (B) > 0.$ The reduction to this case is straightforward and we omit it. In this proof, we allow $\lesssim$ to depend on $\cH^s_\infty(A), \cH^t_\infty (B).$ For all $\de>0,$ by the discrete Frostman's Lemma \cite[Lemma 3.13]{fassorp} we find $(\de,s)$-KT $A_0 \subset A$ with $\# A_0 \sim \de^{-s},$ and $(\de,t)$-KT $B_0 \subset B$ with $\# B_0 \sim \de^{-t}.$ Apply Proposition \ref{prop.elecont} to $A_0$ and $B_0$ to obtain that
		\begin{equation}
			\cont{2(s+t)/3}(f(A\times B)) \geq \cont{2(s+t)/3}(f(A_0\times B_0)) \gtrsim \de^{O(\e)}.
		\end{equation}
		In particular, for all $\de > 0$ we have
		\begin{equation}
			\cont{2(s+t)/3-O(\e)}(f(A\times B)) \gtrsim 1.
		\end{equation}

		Now let $\eta > 0$ and let $\cU$ be a cover of $f(A\times B)$ by disjoint intervals, and such that
		\begin{equation}
			\sum_{U \in \cU} \diam(U)^{2(s+t)/3 - O(\e)} \leq \cH_{\infty}^{2(s+t)/3 - O(\e)}(f(A\times B)) + \eta.
		\end{equation}
		Let $\de >0$ be the diameter of the smallest in the cover, which we may find by the compactness of $f(A\times B).$ Then for this $\de$ we have
		\begin{equation}
			1 \lesssim \cont{2(s+t)/3 - O(\e)}(f(A\times B))\leq \cH_{\infty}^{2(s+t)/3  - O(\e)}(f(A\times B)) + \eta.
		\end{equation}
		Provided that $\eta >0$ is small enough in terms of $\cH^s_\infty(A), \cH^t_\infty (B),$ then 
		\begin{equation}
			\cH_{\infty}^{2(s+t)/3  - O(\e)}(f(A\times B)) > 0.
		\end{equation}
		In particular, 
		\begin{equation}
			\dimh f(A\times B) \geq 2(s+t)/3  - O(\e).
		\end{equation}
		Since $\e > 0$ was arbitrary the result follows. 
	\end{proof}
	\begin{proof}[Proof of Proposition \ref{prop.elecont}]
		Without loss of generality, suppose that $\de > 0$ is dyadic, so find $n \in \N$ so that $\de = 2^{-n}.$ We use the same reduction of step 1 in the proof of Proposition \ref{prop.ele} to replace $A$ and $B$ with dense $\e$-uniform subsets (and thus replacing $\cP$ with another set which is now dense in the new $A \times B).$ Let $\cU$ be a cover of $f(A \times B)$ by dyadic intervals of length at least $\de.$ For $1 \leq j \leq n$ write
		\begin{equation}
			\cU_j := \{U \in \cU: \diam(U) = 2^{-j}\}.
		\end{equation}
		Since
		\begin{equation}
			\sum_{i} \#( f^{-1}(\cup\cU_i)\cap (A \times B)) = \#A\#B,
		\end{equation}
		there exists an index $1 \leq j \leq n$ so that 
		\begin{equation}\label{eq.Pbig}
			\#( f^{-1}(\cup\cU_j)\cap (A \times B)) \gtrapprox \# A \# B.
		\end{equation}
		Select $\rho := 2^{-j},$ and 
		\begin{equation}
			\cP := f^{-1}(\cup\cU_j)\cap (A \times B).
		\end{equation}
		Replace $\cP$ with a dense $\e$-uniform subset. Using the uniformity of $A,B$ and $\cP,$ and \eqref{eq.Pbig} we have 
		\begin{equation}
			\cns\rho\cP N_\cP \gtrsim \de^{O(\e)}\cns\rho A \cns\rho B N_AN_B
		\end{equation}
		where $N_A,N_B,N_\cP$ are the generic values of $\#(A \cap I), \#(B \cap J), \#(\cP \cap Q)$ respectively, for $I \in \cD_\rho(A), J \in \cD_\rho(B), Q \in \cD_\rho(\cP),$ respectively. Since $\cP \subset A \times B$ clearly $N_AN_BN_\cP^{-1} \geq 1.$ Therefore
		\begin{equation}\label{eq.Pbig2}
			N_\rho(\cP) \gtrsim \de^{O(\e)}N_\rho(A)N_\rho(B).
		\end{equation}
		Let $A_0, B_0, \cP_0$ be maximal $\rho$-separated subsets of $A,B,\cP$ respectively. The first two are $(\rho,s,\de^{-O(\e)})$ and $(\rho,t,\de^{-O(\e)})$-sets respectively. This may be observed since $A$ and $B$ are $\e$-uniform. We apply Proposition \ref{prop.ele2} to obtain,
		\begin{equation}
			\cns {\rho}{f(\cP)} \gtrsim \de^{O(\e)} \rho^{-2(s+t)/3}.
		\end{equation}
		In particular, 
		\begin{equation}
			\sum_{U \in \cU} \diam(U)^{-2(s+t)/3} \geq \cns {\rho}{f(\cP)} \rho^{2(s+t)/3} \gtrsim \de^{O(\e)},
		\end{equation}
		completing the proof.
	\end{proof}
 \section{Proof of Theorem \ref{huufwf[gmiu9i96=y]new}}

   We recall the following lemma from \cite{doprod}.
\begin{lemma}
		\label{2endsredu}	
		Let $T$ be a $\delta$-tube with a shading $Y(T)$ that is a $(\delta,s)$-KT set with cardinality $P\le \delta^{-s}$. Let $\epsilon>0$. 
		
		There is a segment $\tau$ of $T$ of length $L$ at least $(\delta^sP)^{\frac1{s-\epsilon^2}}$, that contains  $N\ge L^{\epsilon^2}P$ of the squares in $Y(T)$, call them $Y(\tau)$, and such that 
		\begin{equation}
			\label{pojfiruegioutg}
			\#(Y(\tau)\cap B(x,L(\delta/L)^{\epsilon}))\le (\delta/L)^{\epsilon^3}N,\;\text{for each }x.
		\end{equation}
        In other words $Y(\tau)$ is $(\e,\e^3)$-two-ends.
	\end{lemma}
	\bigskip
    
	We will denote by $G=(A\sqcup B,E)$ a typical bipartite graph with vertex set $V(G)=A\sqcup B$ and edge set $\E(G)=E$ consisting of edges between vertices in $A$ and vertices in $B$.
	$N_G(v)$ will refer to the neighbors of $v$ in $G$ and $\deg_G(v)=\#N_G(v)$ to its degree.
	
	The following result will come in handy.  
	\begin{lemma}[\cite{Zego}]
		\label{jdjjhvjiojrtihoprykhiytih}
		Let $G=(A\sqcup B,E)$ be a bipartite graph. Then there exists $A'\subset A$, $B'\subset B$ such that if we consider the induced graph $(A'\sqcup B',E')$ then
		
		1. For each $a\in A'$,  $\deg_{G'}(a)\ge\frac{\# E}{4\#A}$
		
		2. For each $b\in B'$,  $\deg_{G'}(b)\ge \frac{\# E}{4\#B}$
		
		3. $\# E'\ge \frac{\# E}{2}$.
	\end{lemma}
	We are now ready to prove Theorem \ref{huufwf[gmiu9i96=y]new} which we recall below for the reader's convenience. 
	\begin{theorem}
		\label{huufwf[gmiu9i96=y]}
		Assume $0\le s\le \frac23$.	
		Consider a collection $\T$ of $\delta$-tubes and a collection $\PP$ of $\delta$-squares such that both $\T$ and $\PP$ are $(\delta,2s)$-KT sets, and both $Y(T)=\{p\in\PP:\;p\cap T\not=\emptyset\}$ and $Y'(p)=\{T\in\T:\;p\cap T\not=\emptyset\}$ 
		are $(\delta,s)$-KT sets, for each $T\in\T$ and each $p\in\PP$. Write
		$$I(\T,\PP)=\sum_{T\in\T}\#Y(T)=\sum_{p\in\PP}\#Y'(p).$$
		Then if $s\le \frac12$
		\begin{equation}
			\label{ orjifu54it0-96-py=56-56}
			I(\T,\PP)\les \delta^{-\frac{3s}{4}}(\#\T\#\PP)^{1/2},
		\end{equation}
		while if $s>\frac12$
		\begin{equation}
			\label{ orjifu54it0-96-py=56-562}I(\T,\PP)
			\les \delta^{-(s-\frac{s^2}{2})}(\#\T\#\PP)^{1/2}+\delta^{-\frac{s}{2}}(\#\T+\#\PP).
		\end{equation}
	\end{theorem}
	\begin{proof}
		Fix $\epsilon>0$.
		There will be various losses of the form $(1/\delta)^{O(\epsilon)}$ that will be hidden in the notation $\les$ and $\approx$. We use pigeonholing to enforce various simultaneous uniformity assumptions. These will only cost losses of order $\les 1$. 
		\bigskip

		We introduce the graph $G$ with vertex set $V(G)=\T\sqcup \PP$ and edge set $\E(G)$ consisting of all unordered pairs $(T,p)$ with $p\cap T\not=\emptyset$. Note that $N_G(T)=Y(T)$, $N_G(p)=Y'(p)$ and $\#\E(G)=I(\T,\PP)$. This graph will undergo several transformations, both refinements and contractions.
		\\
		\\
		Step 1. (Parameters $N,L$, the graph $G_1$) We identify a subset  $\T_1$ of $\T$ with the following properties. First, we may  assume all $T\in\T_1$	have the parameter $\#N_{G}(T)$ in the same dyadic range, $\#N_G(T)\sim P$. We choose $P$ such that the number of incidences for this smaller collection is comparable (within a multiplicative factor of $\log \delta^{-1}$) with the number of original incidences. This is the first example of an incidence-preserving refinement, and more will follow.  
		
		We apply Lemma \ref{2endsredu} to each $T\in\T$ to produce the segment $\tau_T$ of length $L_T$, with an $(\epsilon,\epsilon^3)$-two-ends shading $Y(\tau_T)$. This means that after rescaling by $L_T^{-1}$, the shading becomes $(\epsilon,\epsilon^3)$-two-ends at scale $\delta/L_T$. We may also make another refinement to assume that $L_T\sim L$, $\#Y(\tau_T)\sim N$ for all $T\in\T_1$.
		In particular,  $N\approx P$ and $L\ges \delta N^{1/s}$.  
		
		We refine $G$ accordingly, to obtain a new graph $G_1$. First, the vertex set is $V(G_1)=\T_1\sqcup \PP$. Second, we only keep the edge between $p\in\PP$ and $T\in\T_1$ (which in our general notation can be written as either $p\in N_{G_1}(T)$ or $T\in N_{G_1}(p)$) if $p\in Y(\tau_T)$. Note that
		\begin{equation}
			\label{l jruu r fi5thug 0-5to g0-67ih 0-6i}
			\#\E(G_1)\approx \#\E(G)\approx N\#\T_1.
		\end{equation}  
		\\
		\\
		Step 2. (Parameters $N',L'$, the graph $G_2$) We repeat the argument on the dual side. We restrict attention to all $p$ in some subset of $\PP$, that have the  parameter $\#N_{G_1}(p)$ in the same dyadic range, $\#N_{G_1}(p)\sim P'$. The choice of $P'$ is enforced by the same incidence-preserving mechanism.
		
		We apply Lemma \ref{2endsredu} to each such  $p$ to produce an arc $\theta_p\subset\cS^1$ of length $L_p'$ such that the collection -- denoted by $N_{G_1}(\theta_p)$ -- of those tubes  $T\in N_{G_1}(p)$ with directions $d(T)$ inside $\theta_p$ is $(\epsilon,\epsilon^3)$-two-ends. What this means is that these directions, after being rescaled by $1/L_p'$, form an $(\epsilon,\epsilon^3)$-two-ends subset of $\cS^1$ at scale  $\delta/L_p'$.
		
		We may also assume that $L_p'\sim L'$, $\#N_{G_1}(\theta_p)\sim N'$ for all $p$. Our construction gives $N'\approx P'$, $L'\ges \delta {N'}^{1/s}$.
		Call $\PP_1$ the subset of $\PP$ that survived these refinements.
		
		We refine $G_1$ accordingly, to obtain a new graph $G_2$. First, the new vertex set is $V(G_2)=\T_1\sqcup\PP_1$. Second, we only keep the edge between $p\in\PP_1$ and $T\in\T_1$ (which in our general notation can be written as either $p\in N_{G_2}(T)$ or $T\in N_{G_2}(p)$) if $T\in N_{G_1}(\theta_p)$. Note that 
		\begin{equation}
			\label{l jruu r fi5thug 0-5to g0-67ih 0-6i2}
			\#\E(G_1)\approx \#\E(G_2)\approx N'\#\PP_1.
		\end{equation}

		The parameters $P,P'$ will not be mentioned again, as they are replaced with $N,N'$.
        Since $Y(T)$ and $Y'(p)$ are $(\delta,s)$-KT sets, we have
        $$N,N'\lesssim\delta^{-s}.$$

        If
        $$\min(N,N')\les \delta^{-s/2},$$
        we are done. Indeed, when $s\le 1/2$, the preceding bound gives $NN'\les\delta^{-3s/2}$, and hence
        $$I(\T,\PP)\les (N\#\T N'\#\PP)^{1/2}\les \delta^{-\frac{3s}{4}}(\#\T\#\PP)^{1/2}.$$
        When $s>1/2$, if $N\les\delta^{-s/2}$, then \eqref{l jruu r fi5thug 0-5to g0-67ih 0-6i} gives
        $$I(\T,\PP)\approx N\#\T_1\les\delta^{-s/2}\#\T.$$
        Similarly, if $N'\les\delta^{-s/2}$, then \eqref{l jruu r fi5thug 0-5to g0-67ih 0-6i2} gives
        $$I(\T,\PP)\approx N'\#\PP_1\les\delta^{-s/2}\#\PP.$$
        This accounts for the additional term $\delta^{-s/2}(\#\T+\#\PP)$ in \eqref{ orjifu54it0-96-py=56-562}.

		We may now assume that the opposite holds. This means that there is a constant $\gamma$ such that
		$$
		N,N'\gtrsim \delta^{-\frac{s}2-\gamma}.
		$$
		
		When combined with $L\ges \delta N^{1/s}$, $L'\ges \delta {N'}^{1/s}$ this implies that for some constant $\beta>0$ we have 
		\begin{equation}
			\label{kouighruiguyyyguyguyguiytuir}
			L,L'\gtrsim \delta^{\frac12-\beta}.
		\end{equation}
		This assumption will come handy later.
		\\
		\\
		Step 3. (Parameter $M$, the graph $G_3$) We let $\cT$ be the collection of all distinct segments $\tau_T$, for $T\in \T_1$. We refine $\cT$ to a new collection $\cT_2$ and assume each $\tau\in\cT_2$ coincides ( in the sense that $|\tau \cap \tau_T| \geq |\tau|/2$, where $|\cdot |$ is the area measure) with $\tau_T$ for $\sim M$ many $T\in \T_1$. 
		This amounts to another refinement $\T_2\subset \T_1$, as we only keep the tubes $T$ for which $\tau_T\in\cT_2$. All tubes in $\T_2$ are grouped into clusters of size $\sim M$, with each cluster containing a unique $\tau\in \cT_2$. Note that $$M\#\cT_2\sim \#\T_2.$$
		Since $Y'(p)$ was assumed to be $(\delta,s)$-KT, it follows that $M\lesssim L^{-s}$. This bound will not be used, but should help in understanding the new parameter.

		The refined graph $G_3$ will have $V(G_3)=\T_2\sqcup \PP_1$ and the induced edges. The choice of $M$ is made so that
		$$\#\E(G_3)\approx \#\E(G_2).$$
		\\
		\\
		Step 4. (Contracting $G_3$ to $G_4$) Assume  $T,T'\in\T_2$ are in the same cluster ($\tau_T=\tau_{T'}$). We claim that $N_{G_3}(T)=N_{G_3}(T')$. Indeed, if $p\in N_{G_3}(T)$, it follows that $d(T)\in\theta_p$. Note that \eqref{kouighruiguyyyguyguyguiytuir} implies that  $\angle(T,T')\lesssim \delta/L\ll l(\theta_p)\sim L'$. Thus, subject to a negligible enlargement of $\theta_p$, we may assume that also $d(T')\in \theta_p$. Given this, none of the previous refinements $G\to G_1\to G_2\to G_3$ would have removed an edge between $p$ and $T'$ from $\E(G)$, unless it also removed the edge between $p$ and $T$.
		
		In light of this, it is helpful to recast $G_3$ as a new graph $G_4$ with vertex set $V(G_4)=\cT_2\sqcup \PP_1$ and an edge between $\tau\in \cT_2$ and $p\in \PP_1$ if and only if $\E(G_3)$ contains an edge between some (any, according to the claim just proved) $T\in \T_2$ and $p$. It is easy to see that for each $p\in\PP_1$, $N_{G_4}(p)$ coincides with those $\tau\in \cT_2$
		such that $p\cap \tau\not=\emptyset$ and $d(\tau)\in \theta_p$.

		Note that
		$$\#\E(G_4)\sim M^{-1}\#\E(G_3).$$
		\\
		\\
		Step 5. (Squares $Q$) We partition $[0,1]^2$ into the collection $\cQ$ of squares $Q$ with side length $L$. Call $\cT_Q$ those $\tau\in \cT_2$ lying inside/intersecting $Q$. We call $\PP_Q$ those $p\in\PP_1$ lying inside $Q$.
		\\
		\\
		Step 6. (Rectangles $R$, segments $\gamma$)
		We fix a maximal $L'$-separated set of directions $D_{L'}\subset \cS^1$.
		For each $Q$ and for each direction in $D_{L'}$, we tile/cover $Q$ with $\sim 1/L'$ many $(LL',L)$-rectangles $R$. There are $\sim (L')^{-2}$ many such $R$ for each $Q$, call them $\cR_Q$. It is worth noting that the eccentricity of each $R$ coincides with the length of an arc $\theta_p$. Let $\cR=\cup_{Q\in\cQ}\cR_Q$.
		
		Each  $\tau\in \cT_Q$ fits inside a unique $R$. We call  $\cT_R$ the collection of these $\tau$, and note that 
		$$\sum_{R\in\cR}\#\cT_R=\#\cT_2.$$
		Simple geometry shows that if $\tau,\tau'\in\cT_R$ have nonempty intersection, it must be (roughly speaking) a segment with dimensions $(\delta,k\delta/L')$, for some $k\in\N$. It thus makes sense to tile $R$ with segments $\gamma$ with dimensions $(\delta,\delta/L')$, having the same orientation and eccentricity as $R$. Each $\tau\in\cT_Q$ may be thought of as the union of such $\gamma$.
		\\
		\\
		Step 7. (The squares $\PP_R$) For $R\in\cR$, we call $\PP_R$ the collection of those $p\in \PP_1$ that lie inside $R$, and whose arc $\theta_p$ contains the direction $d(R)$  of $R$.
		
		Thus, the sets $(\PP_R)_{R\in\cR}$ are pairwise disjoint, which implies that
		$$\sum_{R\in\cR}\#\PP_R=\#\PP_1.$$
		We note that for each $p\in\PP_R$,   $N_{G_4}(p)$ coincides with those $\tau\in\cT_R$ such that $p\cap \tau\not=\emptyset$.
		This is because the eccentricity and orientation of $R$ match the length $L'$ and location of the arc $\theta_p$. The graph $G_4$ is the union of pairwise disjoint subgraphs, one for each $R\in\cR$. This will allow us to perform the incidence count separately for each $R$.
		\\
		\\
		Step 8. (Parameter $M'$,  the graph $G_5$)
		Each $p\in \PP_R$ lies inside a unique $\gamma\subset R$. It follows that 
		$$\#\E(G_4)\sim \sum_{R}\sum_{M'}\sum_{\gamma\subset R\atop{\#(\PP_R\cap \gamma)\sim M'}}\sum_{p\in\PP_R\cap \gamma}N_{G_4}(p).$$
		We pick $M'$ such that
		$$\#\E(G_4)\approx \sum_{R}\sum_{\gamma\in \Gamma_R}\sum_{p\in\PP_R\cap \gamma}N_{G_4}(p).$$
		where $$\Gamma_R=\{\gamma\subset R:\;\#(\PP_R\cap \gamma)\sim M'\}.$$
		Since $Y(T)$ was assumed to be $(\delta,s)$-KT, it follows that $M'\lesssim {L'}^{-s}$. However, this bound will not be used.

		The triple $(N',L',M')$ will play a similar role to the triple $(N,L,M)$, due to the point-line duality. 
		
		The resulting refinement $G_5$ of $G_4$ will have vertex set $V(G_5)=\cT_2\sqcup \left(\cup_R\cup_{\gamma\in \Gamma_R}(\PP_R\cap \gamma)\right)$, and the induced edges. We have
		$$\#\E(G_5)\approx \#\E(G_4).$$
		Write
		$\Gamma=\cup_R\Gamma_R.$
		\\
		\\
		Step 9. (Contracting $G_5$ to $G_6$)
		Note that for  each $p\in\PP_R\cap \gamma$ with $\gamma\in\Gamma_R$
		\begin{equation}
        \label{ckljhvuihfvjirui9yih097i}
        N_{G_5}(p)=\{\tau\in\cT_R:\;\gamma\subset \tau\}.
        \end{equation}
        Indeed, from a geometric standpoint, this collection contains precisely those $\tau$ such that $d(\tau)\in\theta_p$. And as observed at the end of Step 4, these coincide with $N_{G_4}(p)$. Since $\gamma\in\Gamma_R$, we have 
        $$N_{G_5}(p)=N_{G_4}(p).$$
		Now, \eqref{ckljhvuihfvjirui9yih097i} shows that $N_{G_5}(p)$ takes the same value when $p\in\PP_R\cap \gamma$. Thus, we may recast $G_5$ as a new graph $G_6$ with vertex set $V(G_6)=\cT_2\sqcup\Gamma$, and an edge between $\gamma$ and $\tau$ if and only if $\gamma\subset \tau$. Note that
		$$\#\E(G_6)\approx (M')^{-1}\#\E(G_5).$$ 
		
		Let us summarise the key properties of the  graph $G_6$:
		\begin{equation}
			\label{cj iojrtioguruythu9ytih9yi091}
			\#\E(G_6)\approx \frac{N'\#\PP_1}{MM'}\approx\frac{N\#\T_1}{MM'}\end{equation}
		\begin{equation}
			\label{cj iojrtioguruythu9ytih9yi092}
			\deg_{G_6}(\gamma)\lesssim \frac{N'}{M}\;\;\forall \gamma\in\Gamma,\;\;\deg_{G_6}(\tau)\lesssim \frac{N}{M'}\;\;\forall\tau\in\cT_2
		\end{equation}
		\begin{equation}
			\label{cj iojrtioguruythu9ytih9yi093}
			\#\Gamma\lesssim \frac{\#\PP_1}{M'},\;\;\#\cT_2\lesssim \frac{\#\T_1}{M}.
		\end{equation}
		Note that \eqref{cj iojrtioguruythu9ytih9yi091}
		and \eqref{cj iojrtioguruythu9ytih9yi092} force the opposite inequalities to \eqref{cj iojrtioguruythu9ytih9yi093} to also hold, so in fact
		\begin{equation}
			\label{cj iojrtioguruythu9ytih9yi094}
			\#\Gamma\approx \frac{\#\PP_1}{M'},\;\;\#\cT_2\approx \frac{\#\T_1}{M}.
		\end{equation}
		\\
		\\
		Step 10. (The graph $G_7$)
		We apply Lemma \ref{jdjjhvjiojrtihoprykhiytih} to $G_6$ and find $\Gamma'\subset \Gamma$ and $\cT'\subset\cT_2$ such that, invoking \eqref{cj iojrtioguruythu9ytih9yi094}, the induced graph $G_7$ with vertex set $V(G_7)=\cT'\sqcup \Gamma'$ satisfies 
		\begin{equation}
			\label{c iorj guirtig rohu[op rtgporti]6}
			\deg_{G_7}(\gamma)\approx \frac{N'}{M}\;\;\forall \gamma\in\Gamma',\;\;\deg_{G_7}(\tau)\approx \frac{N}{M'}\;\;\forall\tau\in\cT'
		\end{equation}
		and 
		$$\#\E(G_7)\sim \#\E(G_6).$$
		It also follows that 
		\begin{equation}
			\label{c iorj guirtig rohu[op rtgporti]}
			\#\Gamma'\approx \frac{\#\PP_1}{M'},\;\;\#\cT'\approx \frac{\#\T_1}{M},\text{ and thus }
			M\#\PP_1\#\cT'\approx M'\#\T_1\#\Gamma'.\end{equation}
		Write $\Gamma_R'$ for those $\gamma\in\Gamma'$ lying inside $R$ and write $\cT_R'$ for those $\tau\in\cT'$ lying inside $R$.
		\\
		\\
		Step 11. (Rescaling $R$) Let $\bar{\delta}=\frac{\delta}{LL'}$. Fix  $R$. 
		We rescale it by $(\frac{1}{LL'},\frac1L)$ to get $[0,1]^2$. Then each $\tau\in\cT_R'$ becomes a $\bar{\delta}$-tube, and we call this collection $\bar{\T}_R$. Also, each $\gamma\in\Gamma_R'$ becomes a $\bar{\delta}$-square, and we call this collection $\bar{\PP}_R$. Then $\#\bar{\PP}_R=\#\Gamma_R'$ and $\#\bar{\T}_R=\#\cT_R'$. We note that \eqref{kouighruiguyyyguyguyguiytuir} implies that $\bar{\delta}\ll 1$. In  the new notation \eqref{c iorj guirtig rohu[op rtgporti]} becomes \begin{equation}
			\label{jtgi5iyi670uoi97i0-67iu0-i}
			M\#\PP_1\sum_{R\in\cR}\#\bar{\T}_R\approx M'\#\T_1\sum_{R\in\cR}\#\bar{\PP}_R.
		\end{equation}
		It will be important to realise what this rescaling does: it stretches each tube segment $\tau\in\cT_R'$ by a factor of $1/L$ and it magnifies angles between tube segments by a factor of $1/L'$.
		\\
		\\
		Step 12. (Preservation of spacing properties  via rescaling) The following properties can be easily seen to be true for each $R$.
		
		First, $\bar{\PP}_R$ is a $(\bar{\delta},2s,K_1')$-KT set with $K_1'\sim \min(\frac{1}{M'{L'}^{2s}},\frac1{M'L'})$. Indeed, consider a ball $B(x,r)$ in the rescaled coordinates, where $r\gtrsim\bar{\delta}$. A fixed enlargement of its preimage is a box with dimensions comparable to $(rLL',rL)$, and contains every $\gamma$ whose rescaled square meets $B(x,r)$. Each $\gamma\in\Gamma_R'$ contains $\sim M'$ points of $\PP_R$. Using the $(\delta,2s)$-KT property of $\PP$, we may either place this box inside a square with side length comparable to $rL$, or cover it by $\lesssim 1/L'$ squares with side length comparable to $rLL'$. Consequently,
        \begin{align*}
        M'\#(\bar{\PP}_R\cap B(x,r))
        &\lesssim \min\left\{\left(\frac{rL}{\delta}\right)^{2s},
        \frac{1}{L'}\left(\frac{rLL'}{\delta}\right)^{2s}\right\}\\
        &=\left(\frac{r}{\bar{\delta}}\right)^{2s}
        \min\left\{\frac{1}{{L'}^{2s}},\frac{1}{L'}\right\}.
        \end{align*}
        This gives the stated value of $K_1'$.
        
		Second, a similar reasoning shows that $\bar{\T}_R$ is a $(\bar{\delta},2s,K_1)$-KT set, with $K_1\sim \min(\frac{1}{ML^{2s}},\frac{1}{ML})$.
		
		Third, for each $\bar{T}\in\bar{\T}_R$ the shading $\bar{Y}(\bar{T})=\bar{T}\cap \bar{\PP}_R$ is a $(\bar{\delta},s,K_2)$-KT set, with $K_2\sim \frac{1}{M'{L'}^s}$. Moreover, by \eqref{c iorj guirtig rohu[op rtgporti]6},
		$\#\bar{Y}(\bar{T})=\deg_{G_7}(\tau)\approx \frac{N}{M'}$, where $\tau\in\cT'$  gets rescaled to $\bar{T}$.
		
		Fourth, for each $\bar{p}\in\bar{\PP}_R$, the (dual) shading $\bar{Y}'(\bar{p})=\{\bar{T}\in\bar{\T}_R:\;\bar{p}\cap \bar{T}\not=\emptyset\}$ is a $(\bar{\delta},s,K_2')$-KT set, with $K_2'\sim \frac{1}{M{L}^s}$. Moreover,
		$\#\bar{Y}'(\bar{p})\approx \frac{N'}{M}$, again by \eqref{c iorj guirtig rohu[op rtgporti]6}.
		\\
		\\
		Step 13. (Two-ends preservation via rescaling) We prove that $\bar{Y}(\bar{T})$ is $(\alpha,\epsilon^4)$-two-ends with respect to the scale $\bar{\delta}$, where $\alpha\sim \epsilon$. A similar argument (that we omit) will prove a similar statement about $\bar{Y}'(\bar{p})$. 
		
		Call $\tau$ the segment in $\cT'$ that gives rise to $\bar{T}$ via rescaling. 
		Choose $\alpha$ such that $\bar{\delta}^\alpha=(\delta/L)^\epsilon$. Since $L,L'\ges \delta^{1/3}$ (due to \eqref{kouighruiguyyyguyguyguiytuir}), we see that $\alpha\sim \epsilon$. We consider an arbitrary segment of $\bar{T}$ of length $\bar{\delta}^{\alpha}$ and prove that it contains -- call this number $U$ -- at most a $\bar{\delta}^{\epsilon^4}$-fraction of $\bar{Y}(\bar{T})$. Via rescaling, this number is at most  $1/M'$ times the number -- call it $V$ -- of $\delta$-squares from $\PP$ in some segment of $Y(\tau)$
		of length $\sim \bar{\delta}^{\alpha}L$. Due to our choice of $\alpha$, this length coincides with $(\delta/L)^\epsilon L$. Our construction of $Y(\tau)$ is aligned with \eqref{pojfiruegioutg}, thus 
		$$V \le (\delta/L)^{\epsilon^3}\#Y(\tau)\sim(\delta/L)^{\epsilon^3}N\le \bar{\delta}^{\epsilon^3}N.$$ 
		This implies the desired estimate (the last inequality is true for small $\delta$, since $\bar{\delta}\lesssim \delta^{1/3}$) 
		$$U\le \bar{\delta}^{\epsilon^3}\frac{N}{M'}\approx \bar{\delta}^{\epsilon^3}\#\bar{Y}(\bar{T})\implies U\le \bar{\delta}^{\epsilon^4}\#\bar{Y}(\bar{T}) . $$	
		
		Step 14. (First estimate for $s\le \frac12$) We apply Theorem \ref{completeWW} to the pair $(\bar{\T}_R,\bar{Y})$. Since $s\le \frac23$, we have that $\sigma=\min(2s,2-2s)\ge s$. Thus, by Step 12,  for each $\bar{T}\in\bar{\T}_R$ the shading $\bar{Y}(\bar{T})=\bar{T}\cap \bar{\PP}_R$ is a $(\bar{\delta},\sigma,K_2)$-KT set, with $K_2\sim \frac{1}{M'{L'}^s}$. When $s=\frac12$ we may in fact take $K_2=1$, but this will not lead to better estimates.
		
		We get
		$$\#\bar{\T}_R(\frac{N}{M'})^{3/2}\approx\sum_{\bar{T}\in\bar{\T}_R}\#\bar{Y}(\bar{T})(\frac{N}{M'})^{1/2}\les \bar{\delta}^{-s}\#\bar{\PP}_RK_1K_2^{1/2}\sim \bar{\delta}^{-s}\#\bar{\PP}_R\frac{1}{ML^{2s}}\frac1{{M'}^{1/2}}\frac1{{L'}^{s/2}}.$$	
		When combined with	\eqref{jtgi5iyi670uoi97i0-67iu0-i} and summation in $R$ we get
		$$\frac{M'}{M}\frac{\#\T_1}{\#\PP_1}(\frac{N}{M'})^{3/2}\les (\frac{LL'}{\delta})^s\frac1{L^{2s}{L'}^{s/2}}\frac1{M{M'}^{1/2}},$$
		or equivalently
		\begin{equation}
			\label{jfrtgupgi-46iy9i4-u}
			N^{3/2}\frac{\#\T_1}{\#\PP_1}\les \delta^{-s}\frac{{L'}^{s/2}}{L^s}.
		\end{equation}
		\\
		\\
		Step 15. (Second estimate for $s\le \frac12$) We similarly apply Theorem \ref{completeWW} to the pair $(\bar{\PP}_R,\bar{Y}')$ and get
		\begin{equation}
			\label{jfrtgupgi-46iy9i4-u2}
			{N'}^{3/2}\frac{\#\PP_1}{\#\T_1}\les \delta^{-s}\frac{{L}^{s/2}}{{L'}^s}.
		\end{equation}
		\\
		\\
		Step 16. We take the product of  \eqref{jfrtgupgi-46iy9i4-u} and \eqref{jfrtgupgi-46iy9i4-u2}	
		$$(NN')^{3/2}\les \delta^{-2s}\frac{1}{(LL')^{s/2}}.$$
		Then, recalling that $L\ges \delta N^{1/s}$, $L'\ges \delta {N'}^{1/s}$ we find that 
		$$NN'\les \delta^{-3s/2}.$$
		This together with 
		\eqref{l jruu r fi5thug 0-5to g0-67ih 0-6i} and \eqref{l jruu r fi5thug 0-5to g0-67ih 0-6i2}  gives the desired estimate
		$$I(\T,\PP)\les (NN'\#\T_1\#\PP_1)^{1/2}\les \delta^{-\frac{3s}{4}}(\#\T\#\PP)^{1/2}.$$	
		\\
		\\
		Step 17. (the case $s>\frac12$) We repeat the argument from Steps 14, 15, 16  using the better bounds $K_1\sim \frac1{ML}$ and $K_1'\sim \frac1{M'L'}$. We get $NN'\les \delta^{{s^2}-2s}$, which leads to 
		$$I(\T,\PP)\les \delta^{-(s-\frac{s^2}{2})}(\#\T\#\PP)^{1/2}.$$
        Together with the estimate for $\min(N,N')\les\delta^{-s/2}$ above, this proves \eqref{ orjifu54it0-96-py=56-562}. 	
	\end{proof}
	\begin{remark}
    \label{lkijiourtiouiruythu98}
		The necessity of localising the incidence count (via Theorem \ref{completeWW})  inside rectangles $R\in\cR_Q$ is motivated by our simultaneous application of Lemma \ref{2endsredu} for both tube and  dual shadings. Let $s\le \frac12$. If we only apply this lemma for shadings of tubes and do not use the fact that $Y'(p)$ is $(\delta,s)$-KT, a simpler version of the argument (using incident count for each $Q$) leads to the bound
		$$I(\T,\PP)\les \delta^{-\frac{4s}{5}}(\#\PP)^{2/5}(\#\T)^{3/5}.$$
		Doing the symmetric argument (that is, using Lemma \ref{2endsredu} for dual shadings) gives the bound
		$$I(\T,\PP)\les \delta^{-\frac{4s}{5}}(\#\PP)^{3/5}(\#\T)^{2/5}.$$
		Taking the geometric average gives
		$$I(\T,\PP)\les \delta^{-\frac{4s}{5}}(\#\T\#\PP)^{1/2},$$
		which is weaker than \eqref{ orjifu54it0-96-py=56-56}. When $s>\frac12$, the same argument gives $$I(\T,\PP)\les \delta^{-\frac{2s}{2+s}}(\#\T\#\PP)^{1/2},$$
		which is weaker than \eqref{ orjifu54it0-96-py=56-562}.
		
		It turns out that localising the incidence count inside $R$ rather than $Q$ is more efficient. The cubes $Q$ contain too many tubes and squares, and Theorem \ref{completeWW} becomes lossy when applied to  these collections.

	\end{remark}

\bibliographystyle{alpha}
	\bibliography{references}
    \end{document}